\documentclass[a4paper,12pt,twoside]{article}	
\usepackage{amsmath,amssymb,amsthm,amsfonts,color,graphicx}
\usepackage{setspace}
\usepackage[margin=4cm]{geometry}
\newtheorem{theorem}{Theorem}
\DeclareMathOperator{\hr}{\mathbb H^2\times \mathbb R}
\DeclareMathOperator{\rr}{\mathbb R}

\newtheorem{proposition}{Proposition}
\newtheorem{remark}{Remark}

\theoremstyle{definition}\newtheorem{definition}{Definition}

\numberwithin{equation}{section}
																																																						
\usepackage{layout,textcomp}
\newtheorem*{claim}{Claim}

\title{A half-space theorem for ideal Scherk graphs in $M\times\rr$}
\author{Ana Menezes}
\date{}

\begin{document}
\maketitle
\begin{abstract}
We prove a half-space theorem for an ideal Scherk graph $\Sigma\subset M\times\rr$ over a polygonal domain $D\subset M,$ where $M$ is a Hadamard surface whose curvature is bounded above by a negative constant. More precisely, we show that a properly immersed minimal surface contained in $D\times\rr$ and disjoint from $\Sigma$ is a translate of $\Sigma.$
\end{abstract}

\section{Introduction}
A well known result in the global theory for proper minimal surfaces in the Euclidean $3$-space is the so called \textit{half-space theorem} due to Hoffman and Meeks \cite{HM}, which says that if a properly immersed minimal surface $S$ in $\rr^3$ lies on one side of some plane $P,$ then $S$ is a plane parallel to $P.$ Moreover, they also proved the \textit{strong half-space theorem}: two properly immersed minimal surfaces in $\rr^3$ that do not intersect must be parallel planes.

The problem of giving conditions which force two minimal surfaces of a Riemannian manifold to intersect has received considerable attention, and many people have worked on this subject.

Notice that there is no half-space theorem in Euclidean spaces of dimensions greater than $4$ since there exist rotational proper minimal hypersurfaces contained in a slab. 

Similarly, there exists no half-space theorem for horizontal slices in $\hr$ since rotational minimal surfaces (catenoids) are contained in a slab \cite{NR1, NR2}. However, there are half-space theorems for constant mean curvature (CMC) $1/2$ surfaces in $\hr$ \cite{HRS, NS}. For instance, Hauswirth, Rosenberg, and Spruck \cite{HRS} proved that if $S$ is a properly immersed CMC $1/2$ surface in $\hr,$ contained on the mean convex side of a horocylinder $C,$ then $S$ is a horocylinder parallel to $C;$ and if $S$ is embedded and contains a horocylinder $C$ on its mean convex side, then $S$ is also a horocylinder parallel to $C$. Nelli and Sa Earp \cite{NS} showed that in $\hr$ the mean convex side of a simply connected rotational CMC $1/2$ surface can not contain a complete CMC $1/2$ surface besides the rotational simply connected ones.

Other examples of homogeneous manifolds where there are half-space theorems for minimal surfaces are Nil$_3$ and Sol$_3$ \cite{AR, DH, DMR}. For instance, we know that if a properly immersed minimal surface $S$ in Nil$_3$ lies on one side of some entire minimal graph $\Sigma,$ then $S$ is the image of $\Sigma$ by a vertical translation.

Mazet \cite{M} proved a general half-space theorem for constant mean curvature surfaces. Under certain hypothesis, he proved that in a Riemannian $3$-manifold of bounded geometry, a constant mean curvature $H$ surface on one side of a parabolic constant mean curvature $H$ surface $\Sigma$ is an equidistant surface to $\Sigma.$ 

In this paper, we consider the half-space problem for an ideal Scherk graph $\Sigma$ over a polygonal domain $D\subset M,$ where $M$ denotes a Hadamard surface whose curvature is bounded above by a negative constant, that is, $M$ is a complete simply connected Riemannian surface with curvature $K_M\leq -a^2<0$ for some constant $a\in\rr.$ More precisely, we prove the following result.

\begin{theorem}
Let $M$ denote a Hadamard surface with curvature bounded above by a negative constant, and let $\Sigma=\mbox{Graph}(u)$ be an ideal Scherk graph over an admissible polygonal domain $D\subset M.$ If $S$ is a properly immersed minimal surface contained in $D\times\rr$ and disjoint from $\Sigma,$ then $S$ is a translate of $\Sigma.$
\end{theorem}

We remark that Mazet's theorem does not apply in our case for Scherk surfaces. In fact, in the case of minimal surfaces, one of his hypotheses on the geometry of equidistant surfaces to the parabolic one is that the mean curvature points away from the original surface. However, since an end of a Scherk surface is asymptotic to some vertical plane $\gamma\times\rr,$ where $\gamma$ is a geodesic, we know that an equidistant surface is asymptotic to $\gamma_s\times\rr,$ where $\gamma_s$ is an equidistant curve to $\gamma.$ Hence, in the case of a Scherk surface, the mean curvature vector of an equidistant surface points toward the Scherk surface. 

\section{Preliminaries}
In this section, we present some basic properties of Hadamard manifolds and state some previous results. For more details, see \cite{GR}  or \cite{E1; E2; EO}.


Let $M$ be a Hadamard manifold, that is, a complete simply connected Riemannian manifold with non positive sectional curvature. We say that two geodesics $\gamma_1, \gamma_2$ of $M$, parameterized by arc length, are \textit{asymptotic}  if there exists a constant $c>0$ such that the distance between them satisfies
$$d(\gamma_1(t),\gamma_2(t))\leq c \ \  \mbox{for all} \ t\geq 0.$$ Note that to be asymptotic is an equivalence relation on the oriented unit speed geodesics of $M.$ We call each of these classes a point at infinity. We denote by $M(\infty)$ the set of points at infinity and by $\gamma(+\infty)$ the equivalence class of the geodesic $\gamma.$ Throughout this section, we only consider oriented unit speed geodesics.

Let us assume that $M$ has sectional curvature bounded from above by a negative constant. Then we have two important facts: 
\begin{enumerate}
\item For any two asymptotic geodesics $\gamma_1, \gamma_2,$ the distance between the two curves $\gamma_1|_{\left[t_0\right.,\left.+\infty\right)},\gamma_2|_{\left[t_0\right.,\left.+\infty\right)}$ is zero for any $t_0\in\rr.$
\item Given $x,y\in M(\infty)$, $x\neq y,$ there exists a unique geodesic $\gamma$ such that $\gamma(+\infty)=x$ and $\gamma(-\infty)=y,$ where $\gamma(-\infty)$ denotes the corresponding point at infinity when the orientation of $\gamma$ is changed.
\end{enumerate} 


For any point $p\in M,$ there is a bijective correspondence between the set of unit vectors in the tangent plane $T_pM$ and $M(\infty),$ where a unit vector $v$ is mapped to the point at infinity $\gamma_v(\infty),$ $\gamma_v$ denoting the geodesic with $\gamma_v(0)=p$ and $\gamma_v'(0)=v.$ Analogously, given a point $p\in M$ and a point at infinity $x\in M(\infty),$ there exists a unique geodesic $\gamma$ such that $\gamma(0)=p$ and $\gamma(+\infty)=x.$ In particular, $M(\infty)$ is bijective to a sphere.

There exists a topology on $M^*=M\cup M(\infty)$ satisfying that the restriction to $M$ agrees with the topology induced by the Riemannian distance. This topology is called the cone topology of $M^*$ (see \cite{GR} for instance).

In order to define horospheres, we consider Busemann functions. Given a unit vector $v,$ the Busemann function $B_v: M\rightarrow \rr$ associated to $v$ is defined as
$$
B_v(p)=\lim_{t\rightarrow+\infty}\left(d(p,\gamma_v(t))-t\right).
$$
This is a $C^2$ convex function on $M,$ and it satisfies the following properties.

\textit{Property 1.} The gradient $\nabla B_v(p)$ is the unique unit vector $w$ in $T_pM$ such that $\gamma_v(\infty)=\gamma_w(-\infty).$

\textit{Property 2.} If $w$ is a unit vector such that $\gamma_v(\infty)=\gamma_w(\infty),$ then $B_v-B_w$ is a constant function on $M.$

\begin{definition}
Given a point at infinity $x\in M(\infty)$ and a unit vector $v$ such that $\gamma_v(\infty)=x,$ the \textit{horospheres at $x$} are defined as the level sets of the Busemann function $B_v.$  
\end{definition}

We have the following important facts with respect to horospheres.
\begin{itemize}
\item By \textit{Property 2}, the horospheres at $x$ do not depend on the choice of the vector $v.$
\item The horospheres at $x$ give a foliation of $M,$ and since $B_v$ is a convex function, each bounds a convex domain in $M$ called a \textit{horoball}.
\item The intersection between a geodesic $\gamma$ and a horosphere at $\gamma(\infty)$ is always orthogonal from \textit{Property 1}.
\item Take a point $p\in M$ and let $H_x$ denote a horosphere at $x$. If $\gamma$ is the geodesic passing through $p$ with $\gamma(+\infty)=x$, then $H_x\cap \gamma$ is the closest point on $H_x$ to $p.$
\item Given $x,y\in M(\infty)$, if $\gamma$ is a geodesic with these points at infinity, and $H_x,H_y$ are disjoint horospheres, then the distance between $H_x$ and $H_y$ coincides with the distance between the points $H_x\cap \gamma$ and $H_y\cap \gamma.$
\end{itemize}

From now on, we restrict M to be a Hadamard surface with curvature bounded above by a negative constant, and by horocycle and horodisk we mean horosphere and horoball, respectively.

Let $\Gamma$ be an ideal polygon of $M,$ that is, $\Gamma$ is a polygon all of whose sides are geodesics and the vertices are at infinity $M(\infty).$ We assume that $\Gamma$ has an even number of sides $\alpha_1, \beta_1,\alpha_2,\beta_2,...,\alpha_k,\beta_k.$ Let $D$ be the interior of the convex hull of the vertices of $\Gamma,$ so $\partial D=\Gamma,$ and $D$ is a topological disk. We call $D$ an ideal polygonal domain. 

\begin{definition}
An ideal Scherk graph over $D$ is a minimal surface that is the graph of a function defined on $D$ and taking the values $+\infty$ on each side $\alpha_i$ and $-\infty$ on each side $\beta_i$.
\end{definition}

For the sake of completeness and in order to understand the hypothesis on our main result (Theorem \ref{main-thm}), let us describe the necessary and sufficient conditions on the domain $D,$ proved by G\'alvez and Rosenberg \cite{GR}, for the existence of an ideal Scherk graph over $D$.

At each vertex $a_i$ of $\Gamma,$ place a horocycle $H_i$ so that $H_i\cap H_j=\emptyset$ if $i\neq j.$

Each $\alpha_i$ meets exactly two horodisks. Denote by $\tilde{\alpha_i}$ the compact arc of $\alpha_i$ outside the two horodisks and denote by $|\alpha_i|$ the length of $\tilde{\alpha_i},$ that is, the distance between these horodisks. Analogously, we can define $\tilde{\beta_i}$ and $|\beta_i|.$

Now define
$$
a(\Gamma)=\sum_{i=1}^{k}|\alpha_i|
$$ 
and
$$
b(\Gamma)=\sum_{i=1}^{k}|\beta_i|.
$$

Observe that $a(\Gamma)-b(\Gamma)$ does not depend on the choice of the horocycles because given two horocycles $H_1,H_2$ at a point $x\in M(\infty)$ and a geodesic $\gamma$ with $x$ as a point at infinity, the distance between $H_1$ and $H_2$ coincides with the distance between the points $\gamma\cap H_1$ and $\gamma\cap H_2.$ 

\begin{definition}
An ideal polygon $\mathcal P$ is said to be inscribed in $D$ if the vertices of $\mathcal P$  are among the vertices of $\Gamma.$ Hence, its edges are either interior in $D$ or equal to some $\alpha_i$ or $\beta_j.$
\end{definition}

The definition of $a(\Gamma)$ and $b(\Gamma)$ extends to inscribed polygons:
$$
a(\mathcal P)=\sum_{\alpha_i\in\mathcal P}|\alpha_i| \mbox{\ \ and \ \ } b(\mathcal P)=\sum_{\beta_i\in\mathcal P}|\beta_i|.
$$


We denote by $|\mathcal P|$ the length of the boundary arcs of $\mathcal P$ exterior to the horodisks bounded by $H_i$ at the vertices of $\mathcal P.$ We call this the truncated length of $\mathcal P.$

\begin{definition}
An ideal polygon $\Gamma$ is said to be admissible if the two following conditions are satisfied.
\begin{enumerate}
\item $a(\Gamma)=b(\Gamma);$
\item For each inscribed polygon $\mathcal P$ in $D,$ $\mathcal P\neq\Gamma,$ and for some choice of the horocycles at the vertices, we have
$$
2a(\mathcal P)<|\mathcal P| \mbox{\ \ and \ \ } 2b(\mathcal P)<|\mathcal P|.
$$
\end{enumerate} 
Moreover, an ideal polygonal domain $D$ is said to be admissible if its boundary $\Gamma=\partial D$ is an admissible polygon.
\end{definition}

The properties of an admissible polygon are the necessary and sufficient conditions for the existence of an ideal Scherk graph over $D\subset M$ \cite{GR}. 

An important tool for studying minimal (and more generally, constant mean curvature) surfaces are the formulas for the flux of appropriately chosen ambient vector fields across the surface. 

Let $u$ be a function defined in $D$ whose graph is a minimal surface, and consider $X=\frac{\nabla u}{W}$ defined on $D,$ where $W^2=1+|\nabla u|^2.$ For an open domain $A\subset D$ and $\alpha$ a boundary arc of $A,$ we define the flux formula across $\alpha$ as
$$
F_{u}(\alpha)=\int_{\alpha}\left\langle X, \nu\right\rangle ds;
$$
here $\alpha$ is oriented as the boundary of $A,$ and $\nu$ is the outer conormal to $A$ along $\alpha.$

\begin{theorem}[Flux Theorem]
Let $A\subset D$ be an open domain.
\begin{enumerate}
\item If $\partial A$ is a compact cycle, then $F_{u}(\partial A)=0.$
\item If $\alpha$ is a compact arc of $A,$ then 
$
F_u(\alpha)\leq |\alpha|.
$
\item If $\alpha$ is a compact arc of $A$ on which $u$ diverges to $+\infty,$ then
$$
F_{u}(\alpha)=|\alpha|.
$$ 
\item If $\alpha$ is a compact arc of $A$ on which $u$ diverges to $-\infty,$ then
$$
F_{u}(\alpha)=-|\alpha|.
$$ 
\end{enumerate} 
\end{theorem}

Another usefull result related to the flux formula is the following.

\begin{proposition}
Let $D$ be a domain whose boundary is an ideal polygon, and let $u,v$ be functions defined on $D$ whose graphs are minimal surfaces. If $u\leq v$ on $D$ and $u=v$ on $\partial D,$ then $F_{u}(\partial D)\leq F_v(\partial D).$ Moreover, equality holds if and only if $u\equiv v$ on $D.$ 
\label{prop-flux}
\end{proposition}

To a proof of this result, see, for example, the proof of the generalized maximum principle in \cite{CR}, Theorem 2. 

\section{Main Result}
In this section, we consider a Hadamard surface $M$ whose curvature is bounded above by a negative constant, that is, $M$ is a complete simply connected Riemannian surface with curvature $K_M\leq -a^2<0$ for some constant $a\in\rr.$  We now can establish our main result.

\begin{theorem}
Let $M$ denote a Hadamard surface with curvature bounded above by a negative constant, and let $\Sigma=\mbox{Graph}(u)$ be an ideal Scherk graph over an admissible polygonal domain $D\subset M.$ If $S$ is a properly immersed minimal surface contained in $D\times\rr$ and disjoint from $\Sigma,$ then $S$ is a translate of $\Sigma.$
\label{main-thm}
\end{theorem}

The idea to prove this result is based on the proof of the classical half-space theorem in the Euclidean three-space due to Hoffman and Meeks \cite{HM}. In their proof, they use as barrier a family of minimal surfaces (obtained from the catenoid by homothety) that converges to the plane minus a point, where the plane is the minimal surface for which they want to prove the half-space theorem. Hence, in order to prove our result, using their ideas, we need a family of minimal surfaces that play the role of barriers and converge to our ideal Scherk graph, at least outside a compact set. To construct such a family, we follow an idea of Rosenberg, Schulze, and Spruck \cite{RSS} by constructing a discrete family of minimal graphs in $D\times\rr.$

Let $\Sigma=\mbox{Graph}(u)$ be an ideal Scherk graph over $D$ with $\Gamma=\partial D.$ Given any point $p\in D,$ consider the geodesics starting at $p$ and going to the vertices of $\Gamma.$ Take the points over each one of these geodesics that are at a distance $n$ from $p.$ Now consider the geodesics joining two consecutive points. The angle at which two of these geodesics meet is less than $\pi;$ hence, we can smooth the corners to obtain a convex domain $D_n$ with smooth boundary $\Gamma_n=\partial D_n$ and such that $D_1\subset D_2 \subset \cdots \subset D_n\subset \cdots$ is an exhaustion of $D.$ (See Figure \ref{fig1}.)

\begin{figure}[h]
  \centering
  \includegraphics[height=5cm]{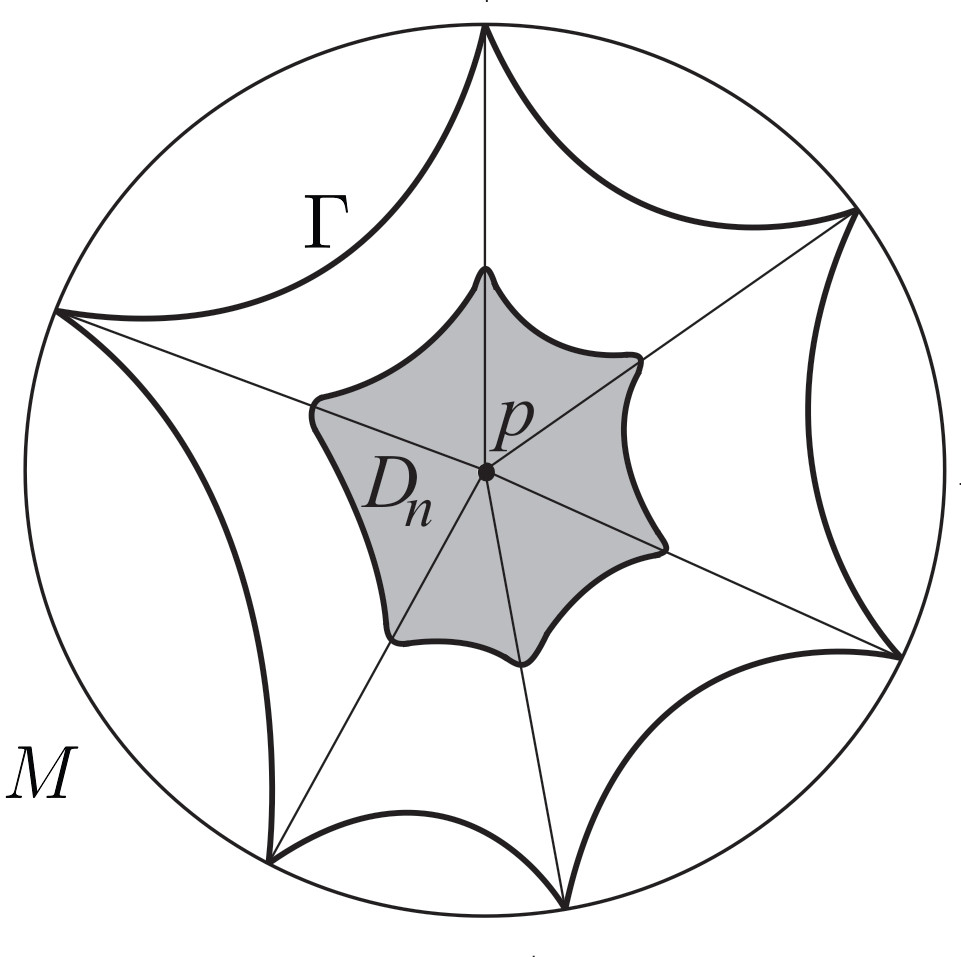}
\caption{Convex smooth domain $D_n$.}
\label{fig1}
\end{figure}



Denote by $A_n$ the annular-type domain $D_n\setminus \bar{D}_1$ and by $\Sigma_n$ the graph of $u$ restrict to $A_n.$ Hence, $\Sigma_n$ is a stable minimal surface, and any sufficiently small perturbation of $\partial \Sigma_n$ gives rise to a smooth family of minimal surfaces $\Sigma_{n,t}$ with $\Sigma_{n,0}=\Sigma_n.$ We use this fact to the deformation of $\partial \Sigma_n$ that is the graph over $\partial A_n$ given by $\partial_1\cup \partial_{n,t}$ for $t\geq0,$ where $\partial_1=\left(\Gamma_1\times\rr\right)\cap \Sigma,$ $\partial_{n,t}=(\Gamma_n\times\rr)\cap T(t)(\Sigma),$ and $T(t)$ is the vertical translation by height $t.$ Then for $t$ sufficiently small, there exists a minimal surface $\Sigma_{n,t}$ that is the graph of a smooth function $u_{n,t}$ defined on $A_n$ with boundary $\partial_1\cup \partial_{n,t}$ (see Figure \ref{fig3}). Note that $u_{n,t}$ satisfies the minimal surface equation on $A_n,$ and, by the maximum principle, $\Sigma_{n,t}$ stays between $\Sigma$ and $\Sigma(t)=T(t)(\Sigma).$ We show that there exists a uniform interval of existence for $u_{n,t},$ that is, we prove that there exists $\delta_0>0$ such that for all $n$ and $0\leq t\leq \delta_0,$ the minimal surfaces $\Sigma_{n,t}=\mbox{Graph}(u_{n,t})$ exist.

\begin{figure}[h]
  \centering
  \includegraphics[height=6cm]{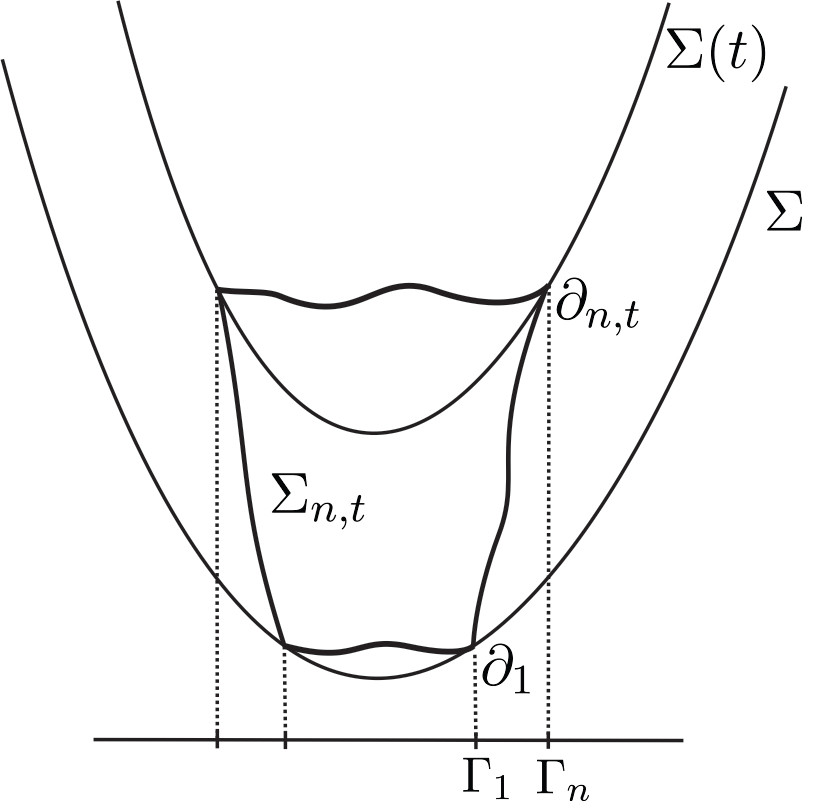}
\caption{Minimal surface $\Sigma_{n,t}.$}
\label{fig3}
\end{figure}

Consider $\delta_0>0$ sufficiently small so that $u_{2,t}$ exists for any $t\in[0, \delta_0].$ We will show this $\delta_0$ works for all $n\geq2,$ that is, we will prove that for $n>2,$ the set $B_n=\{\tau\in[0,\delta_0]; u_{n,t} \ \mbox{exists for} \ 0\leq t\leq \tau \}$ is in fact the interval $[0,\delta_0]$.

\begin{claim}
The set $B_n$ is open and closed in the interval $[0,\delta_0]$. Hence, $B_n=[0,\delta_0].$
\end{claim} 

\begin{proof}
Consider an (increasing) sequence $\tau_k \in B_n$ such that $\tau_k\rightarrow\tau$ when $k\rightarrow\infty.$ First, observe that the family of minimal surfaces $\Sigma_{n,\tau_k}=\mbox{Graph}(u_{n,\tau_k})$ is contained in the region bounded by $\Sigma$ and $\Sigma(\tau),$ in particular, $|u_{n,\tau_k}|\leq C_n$ for all $k$ and that the boundary component $\partial_1$ is contained in the boundary $\partial\Sigma_{n,\tau_k}$ for all $k.$ Then there exists a minimal surface $\Sigma_{n,\tau}$ that is the limit of the surfaces $\Sigma_{n,\tau_k}$ with $\partial_1\subset \partial\Sigma_{n,\tau}.$ It remains to prove that $\Sigma_{n,\tau}$ is a graph. 

Since $D_2\subset D_n,$ using the maximum principle with vertical translations of $u_{2,\delta_0},$ we get that $u_{n,\tau_k}\leq u_{2,\delta_0}$ in a neighborhood of $\Gamma_1.$ Then the gradient of $u_{n,\tau_k}$ is uniformly bounded in a neighborhood of $\Gamma_1.$ We affirm that we also have a uniform bound for points in $\Gamma_n.$ In fact, suppose this is not true, so there exists a sequence $p_k\in \Gamma_n$ with $u_{n,\tau_k}(p_k)\rightarrow p\in\partial_{n,\tau}$ such that $|\nabla u_{n,\tau_k}(p_k)|\rightarrow\infty.$ This implies that the minimal surface $\Sigma_{n,\tau}$ is vertical at $p.$ Considering the horizontal geodesic $\gamma$ that passes through $p$ and is tangent to $\partial_{n,\tau}$ (recall that $\partial_{n,\tau}$ is convex), we can apply the maximum principle with boundary to $\Sigma_{n,\tau}$ and $\gamma\times\left(-\infty\right.,\tau]$ to conclude they coincide, which is impossible. Thus, there exists a constant $C^{'}_n$ such that $|\nabla u_{n,\tau_k}(p)|\leq C^{'}_n$ for all $k$ and $p\in \partial A_n=\Gamma_1\cup \Gamma_n.$

  Since $u_{n,\tau_k}$ are uniformly bounded functions and we have uniform gradient estimates for $u_{n,\tau_k}$ in $\partial A_n,$ we get uniform gradient estimates for $u_{n,\tau_k}$ in the whole domain $A_n$ (for instance, see \cite{DR}, Lemma 2.5). Thus, the DeGiorgi-Nash-Moser and Schauder estimates imply locally uniform estimates for all higher derivatives. Then, using Arzela-Ascoli's theorem, there is some subsequence of $\{u_{n,\tau_k}\}$ that converges to a function $u_{n,\tau}$ defined on $A_n,$ which also satisfies the minimal surface equation; hence, its graph $\hat{\Sigma}=\mbox{Graph}(u_{n,\tau})$ is a minimal surface with boundary $\partial_1\cup\partial_{n,\tau}.$ By uniqueness of the limit of $\Sigma_{n,\tau_k}$ we conclude that $\Sigma_{n,\tau}=\hat{\Sigma}=\mbox{Graph}(u_{n,\tau}).$ Then $\tau\in B_n,$ and $B_n$ is closed.

From the previous discussion we know that, as a consequence of the maximum principle, if $\tau\in B_n,$ then the tangent plane to the boundary of $\Sigma_{n,\tau}=\mbox{Graph}(u_{n,\tau})$ is never vertical. Then $\Sigma_{n,\tau}$ is strictly stable, and, in particular, a sufficiently small perturbation of its boundary $\partial \Sigma_{n,\tau}=\partial_1\cup \partial_{n,\tau}$ to $\partial_1\cup\partial_{n,\tau+t}$ gives rise to a smooth family of minimal surfaces $\Sigma_{n,\tau+t}$ with boundary $\partial_1\cup\partial_{n,\tau+t}.$ Thus, $B_n$ is open.
\end{proof}

Therefore, we have proved that for all $n\geq 2$ and $0\leq t\leq \delta_0,$ there exists a function $u_{n,t}$ defined on $A_n$ such that $\Sigma_{n,t}=\mbox{Graph}(u_{n,t})$ is a minimal surface with boundary $\partial \Sigma_{n,t}=\partial_1\cup \partial_{n,t}.$ 

Fix $t\in\left(0,\right. \left.\delta_0\right].$ For a fixed $n_0,$ consider the sequence $\{u_{n,t}|_{A_{n_0}}\}$ for $n> n_0.$ We already know that $u_{n,t}\leq u_{n_0,t}$ in a neighborhood of $\Gamma_1;$ hence, we have uniform gradient estimates in such a neighborhood. Moreover, since we have uniform curvature estimates for points far from the boundary \cite{S} and $\Gamma_n\not\subset A_{n_0}$ for all $n> n_0,$ we get that the sequence of minimal surfaces $\Sigma_{n,t},$ $n> n_0$, restricted to the domain $A_{n_0}$ has uniform curvature estimates up to the boundary. Then the DeGiorgi-Nash-Moser and Schauder estimates imply locally uniform estimates for all higher derivatives. Thus, there exists a subsequence $\{u_{n_j,t}|_{A_{n_0}}\}$ that converges to a function $\hat{u}_{n_0}$ defined over $A_{n_0}$ whose graph $\hat{\Sigma}_{n_0}$ is a minimal surface with $\partial_1\subset\partial\hat{\Sigma}_{n_0}$ and $u\leq \hat{u}_{n_0}\leq u+t$ over $A_{n_0}.$

Now consider the subsequence $\{u_{n_j,t}\}$ restricted to $A_{2n_0}$ for $n_j>2n_0.$ Using the same argument as before, the sequence $\{u_{n_j,t}|_{A_{2n_0}}\}$ for $n_j>2n_0$ has a subsequence $\{u_{n_{jk},t}|_{A_{2n_0}}\}$ that converges to a function $\hat{u}_{2n_0}$ defined over $A_{2n_0}$ whose graph $\hat{\Sigma}_{2n_0}$ is a minimal surface with $\partial_1\subset\partial\hat{\Sigma}_{2n_0}$ and $u\leq \hat{u}_{2n_0}\leq u+t$ over $A_{2n_0}.$

Since the sequence $\{u_{n_{jk},t}|_{A_{2n_0}}\}$ is a subsequence of $\{u_{n_{j},t}|_{A_{n_0}}\},$ we conclude, by uniqueness of the limit,  that $\hat{u}_{2n_0}=\hat{u}_{n_0}$ in $A_{n_0}.$

We continue this argument to $A_{kn_0}$ for all $k>2,$ obtaining a function $\hat{u}_{kn_0}$ defined on $A_{kn_0}$ whose graph $\hat{\Sigma}_{kn_0}$ is a minimal surface with $\partial_1\subset\partial\hat{\Sigma}_{kn_0},$ $u\leq \hat{u}_{kn_0}\leq u+t$ over $A_{kn_0},$ and $\hat{u}_{kn_0}=\hat{u}_{ln_0}$ in $A_{ln_0}$ for each $1\leq l\leq k.$ Hence, using a diagonal process, we obtain a subsequence of $\{u_{n,t}\}$ that converges to a function $\hat{u}_\infty$ defined over $\Omega=D\setminus \bar{D_1}$ (the limit of the domains $A_n$) whose graph $\hat{\Sigma}_\infty$ is a minimal surface with $\partial \hat{\Sigma}_\infty=\partial_1,$  $u\leq \hat{u}_{\infty} <u+t$ over $\Omega,$ and $\hat{u}_{\infty}=\hat{u}_{kn_0}$ in $A_{kn_0}$ for all $k.$

For simplicity, let us write $\hat{u}$ and $\hat{\Sigma}$ to denote $\hat{u}_\infty$ and $\hat{\Sigma}_\infty.$

Note that, since $u\leq \hat{u}\leq u+t$ over $\Omega,$ the minimal surface $\hat{\Sigma}=\mbox{Graph}(\hat{u})$ assumes the same infinite boundary values at $\Gamma$ as the ideal Scherk graph $\Sigma=\mbox{Graph}(u).$ Consider the restriction of $u$ to $\Omega$ and continue denoting by $\Sigma$ the graph of $u$ restricted to $\Omega.$ We will show that $\Sigma$ and $\hat{\Sigma}$ coincide by analysing the flux of the functions $u, \hat{u}$ across the boundary of $\Omega,$ which is $\Gamma_1\cup \Gamma,$ and using Proposition \ref{prop-flux}.

Let $\alpha_1,\beta_1,\alpha_2,\beta_2, \dots, \alpha_k,\beta_k$ be the geodesic sides of the admissible ideal polygon $\Gamma$ with $u(\alpha_i)=+\infty=\hat{u}(\alpha_i)$ and $u(\beta_i)=-\infty=\hat{u}(\beta_i).$ For each $n,$ consider pairwise disjoint horocycles $H_i(n)$ at each vertex $a_i$ of $\Gamma$ such that the convex horodisk bounded by $H_i(n+1)$ is contained in the convex horodisk bounded by $H_i(n).$ For each side $\alpha_i,$ let us denote by $\alpha_i^n$ the compact arc of $\alpha_i$ that is the part of $\alpha_i$ outside the two horodisks and by $|\alpha_i^n|$ the length of $\alpha_i^n,$ that is, the distance between the two horodisks. Analogously, we define $\beta_i^n$ for each side $\beta_i.$ Denote by $c_i^n$ the compact arc of $H_i(n)$ contained in the domain $D$ and let ${\mathcal P}^n$ be the polygon formed by $\alpha_i^n,\beta_i^n,$ and $c_i^n.$

Since the function $u$ is defined in the interior region bounded by ${\mathcal P}^n$ and ${\mathcal P}^n$ is a compact cycle, by the flux theorem it follows that $F_u({\mathcal P}^n)=0.$ On the other hand, since $u\leq \hat{u},$ we have $F_u({\mathcal P}^n)\leq F_{\hat{u}}({\mathcal P}^n)$ and then $F_{\hat{u}}({\mathcal P}^n)\geq 0.$ Moreover, the flux of $\hat{u}$ across ${\mathcal P}^n$ satisfies 
$$
\begin{array}{rcl}
F_{\hat{u}}({\mathcal P}^n)&=&\sum_{i}F_{\hat{u}}(\alpha_i^n)+\sum_{i}F_{\hat{u}}(\beta_i^n)+\sum_i F_{\hat{u}}(c_i^n)\\
&&\\
&\leq &\sum_i(|\alpha_i^n|-|\beta_i^n|) + \sum_i|c_i^n|.
\end{array}$$

Notice that $|c_i^n|\rightarrow 0$ as $n\rightarrow\infty,$ and, since $\Gamma$ is an admissible polygon, we have $\sum_i|\alpha_i^n|=\sum_i|\beta_i^n|$ for any $n.$ Hence, we conclude 
$$
F_{\hat{u}}({\mathcal P}^n)\rightarrow 0 \mbox{\ \ as \ \ } n\rightarrow\infty.
$$ 

Then $F_u(\Gamma)=\lim_{n\rightarrow\infty}F_u({\mathcal P}^n)=0=\lim_{n\rightarrow\infty}F_{\hat{u}}({\mathcal P}^n)=F_{\hat{u}}(\Gamma).$ 

On the other hand, since ${\mathcal P}^n$ is homotopic to $\Gamma_1,$ it follows that $F_{\hat{u}}(\Gamma_1)=F_{\hat{u}}({\mathcal P}^n)$ for any $n,$ and we conclude that $F_{\hat{u}}(\Gamma_1)=0.$ Analogously (or using the flux theorem as we did for ${\mathcal P}^n$), $F_u(\Gamma_1)=0.$ Therefore, we have proved that the functions $u$ and $\hat{u}$ have the same flux across the boundary $\partial \Omega=\Gamma_1\cup \Gamma.$

Since $\Sigma=\mbox{Graph}(u)$ and $\hat{\Sigma}=\mbox{Graph}(\hat{u})$ are two minimal graphs over $\Omega=D\setminus \bar{D_1}$ such that $u\leq \hat{u}$ on $\Omega,$ $u=\hat{u}$ on $\partial \Omega,$  and $F_u(\partial\Omega)= F_{\hat{u}}(\partial\Omega),$ we conclude, using Proposition \ref{prop-flux}, that $u\equiv \hat{u}$ over $\Omega,$ that is, $\hat{\Sigma}$ is the Scherk graph over $\Omega$ with $\partial \hat{\Sigma}=\partial_1.$

\begin{remark} We have proved that for any $t\in \left(0\right.,\left.\delta_0\right],$ we can get a subsequence of the minimal surfaces $\Sigma_{n,t}$ that converges to a minimal surface $\hat{\Sigma},$ which is the Scherk graph over $D\setminus \bar{D_1}$ with $\partial\hat{\Sigma}=\partial_1$.
\end{remark}

\vspace{0.2cm}

Now we are able to prove the theorem.

\vspace{0.3cm}

\textit{Proof of Theorem \ref{main-thm}.} Since $\Sigma\cap S=\emptyset,$ we can suppose that $S$ is entirely under $\Sigma.$ Pushing down $\Sigma$ by vertical translations, we have two possibilities: either a translate of $\Sigma$ touches $S$ for the first time in an interior point, and then, by the maximum principle, they coincide; or $S$ is asymptotic at infinity to a translate of $\Sigma.$ Let us analyse this last case.

Without loss of generality, assume that $S$ is asymptotic at infinity to $\Sigma.$ We want to prove that, in fact, they coincide. Suppose that this is not true. Then since $S$ is proper, there are a point $p_0\in\Sigma$ and a cylinder $C=B_{\Sigma}(p_0,r_0)\times (-r_0,r_0)$ for some $r_0>0$ such that $S\cap C=\emptyset,$ where $B_{\Sigma}(p_0,r_0)$ is the intrinsic ball centered at $p_0$ with radius $r_0.$ We can assume $r_0$ is less than the injectivity radius of $\Sigma$ at $p_0.$ In our construction of the surfaces $\Sigma_{n,t},$ we can choose the first domain of the exhaustion $D_1$ sufficiently small so that $\partial_1\subset B_{\Sigma}(p_0,\frac{r_0}{2})$ and take $t=\mbox{min}\{\frac{r_0}{2},\delta_0\}.$

Observe that when we translate $\Sigma_{n,t}$ vertically downward by an amount $t,$ the boundaries of the translates of $\Sigma_{n,t}$ stay strictly above $S.$ Thus, by the maximum principle, all the translates remain disjoint from $S.$ We call $\Sigma^{'}_{n,t}$ this final translate with boundary $\partial\Sigma^{'}_{n,t}=\partial^{'}_1\cup \partial^{'}_n,$ where $T(t)(\partial^{'}_1)=\partial_1\subset\Sigma$ and $\partial^{'}_n\subset\Sigma.$ Hence, all the surfaces $\Sigma^{'}_{n,t}$ lie above $S,$ and, as we proved before, there exists a subsequence of $\Sigma^{'}_{n,t}$ that converges to the ideal Scherk graph $\Sigma^{'}$ defined over $D\setminus  \bar{D_1}$ with $T(t)(\Sigma^{'})=\Sigma.$ In particular, we conclude that $S$ lies below $\Sigma^{'},$ which yields a contradiction since we are assuming that $S$ is asymptotic at infinity to $\Sigma.$
\begin{flushright}
$\Box$
\end{flushright}

\begin{flushleft}
\textsc{Instituto Nacional de Matem\' atica Pura e Aplicada (IMPA)}

\textsc{Estrada Dona Castorina 110, 22460-320, Rio de Janeiro-RJ, Brazil}

\textit{Email adress:} anamaria@impa.br
\end{flushleft}

\end{document}